\documentclass[12pt]{amsart}
\usepackage{amssymb, amsmath}
\title{ Limits of Generalized Quaternion Groups}
\author{R. Hobbi}
\address{R. Hobbi: Department of Pure Mathematics,  Faculty of Mathematical
Sciences, University of Tabriz, Tabriz, Iran}
\author{M. Shahryari}
\address{M. Shahryari: Department of Pure Mathematics,  Faculty of Mathematical
Sciences, University of Tabriz, Tabriz, Iran}
\email{mshahryari@tabrizu.ac.ir}

\newtheorem {theorem}{Theorem}

\begin{document}

\maketitle
\begin{abstract}
In the space of marked group, we determine the structure of groups which are limit points of the set of all generalized quaternion groups.
\end{abstract}


\vspace{1cm}

In the space of marked groups, consider the situation, a sequence   of generalized quaternion groups in which  converges to a  marked group $(G, S)$. We will prove that there exists a finitely generated abelian group $A=\mathbb{Z}^l\oplus \mathbb{Z}_{2^k}$, such that
$$
G\cong \frac{\mathbb{Z}_4\ltimes A}{\langle (2, 2^{k-1})\rangle}.
$$
Here, the cyclic group $\mathbb{Z}_4$ acts on $A$ by $x.a=(-1)^xa$. This gives a partial answer to a question of Champetier and Guirardel on the limits of finite groups, \cite{Champ}. Already, Guyot in \cite{Guyot} studied the same problem for the class of dihedral groups. As any dihedral group is a semidirect product of two cyclic groups, determining their limit points is more straightforward than the case of generalized quaternion groups. The generalized quaternion group $Q_{2^n}$ has the standard presentation
$$
\langle x, y| x^{2^{n-1}}=y^4=1, yxy^{-1}=x^{-1}, x^{2^{n-2}}=y^2\rangle,
$$
and in the same time it can be defined as the quotient
$$
Q_{2^n}=\frac{\mathbb{Z}_4\ltimes \mathbb{Z}_{2^{n-1}}}{\langle(2, 2^{n-2})\rangle},
$$
where the action of $\mathbb{Z}_4$ on $\mathbb{Z}_{2^{n-1}}$ is given by $x\cdot a=(-1)^xa$. In this article, the word {\em quaternion} will be used instead of {\em generalized quaternion}. Four facts about quaternion groups will be used in our arguments:\\

1- they are $2$-groups;

2- they have  unique involution;

3- any subgroup of a quaternion group is cyclic or quaternion;

4- the order two subgroup  $\langle(2, 2^{n-2})\rangle$ is central in $\mathbb{Z}_4\ltimes \mathbb{Z}_{2^{n-1}}$.

\section{Basic notions}
The idea of Gromov-Grigorchuk metric on the space of finitely generated groups is proposed by M. Gromov in his celebrated solution to the Milnor's conjecture on the groups with polynomial growth (see \cite{Gromov}). It is extensively studied by Grigorchuk in \cite{Grig}. For a detailed discussion of this metric, the reader can consult \cite{Champ}. Here, we give some necessary basic definitions. A marked group $(G,S)$ consists of a group $G$ and an $m$-tuple of its elements $S=(s_1, \ldots, s_m)$ such that $G$ is generated by $S$. Two marked groups $(G, S)$ and $(G^{\prime}, S^{\prime})$ are {\em the same}, if there exists an isomorphism $G\to G^{\prime}$ sending any $s_i$ to $s^{\prime}_i$. The set of all such marked groups is denoted by $\mathcal{G}_m$. This set can be identified by the set of all normal subgroup of the free group $\mathbb{F}_m$. Since the later is a closed subset of the compact topological space $2^{\mathbb{F}_m}$ (with the product topology), so it is also a compact space. It is easy to see that, if $(N_i)$ is a convergent sequence in this space, then
$$
\lim N_i=\liminf_i N_i=\limsup_i N_i,
$$
where by definition
$$
\liminf_i N_i=\bigcup_{j=1}^{\infty}\bigcap_{i\geq j}N_i, \ \ \limsup_i N_i=\bigcap_{j=1}^{\infty}\bigcup_{i\geq j}N_i.
$$

This space is in fact metrizable: let $B_{\lambda}$ be the closed ball of radius $\lambda$  in $\mathbb{F}_m$ (having the identity as the center) with respect to its word metric. For any two normal subgroups $N$ and $N^{\prime}$, we say that they are in distance at most $e^{-{\lambda}}$, if $B_{\lambda}\cap N=B_{\lambda}\cap N^{\prime}$. So, if $\Lambda$ is the largest of such numbers, then  we can define
$$
d(N, N^{\prime})=e^{-\Lambda}.
$$
This induces a corresponding metric on $\mathcal{G}_m$. To see what is this metric exactly, let $(G, S)$ be a marked group. For any non-negative integer $\lambda$, consider the set of {\em relations} of $G$ with length at most $\lambda$, i.e.
$$
\mathrm{Rel}_{\lambda}(G, S)=\{ w\in \mathbb{F}_m: \| w\|\leq \lambda, w(S)=1\}.
$$
Then $d((G, S),(G^{\prime}, S^{\prime}))=e^{-\Lambda}$, where $\Lambda$ is the largest number such that $\mathrm{Rel}_{\Lambda}(G,S)=\mathrm{Rel}_{\Lambda}(G^{\prime}, S^{\prime})$.  This metric on $\mathcal{G}_m$ is the so called Gromov-Grigorchuk metric. Equivalently, two marked groups $(G, S)$ and $(G^{\prime}, S^{\prime})$ are close, if large enough balls (around identity) in the corresponding marked Cayley graphs of $(G, S)$ and $(G^{\prime}, S^{\prime})$ are isomorphic.

Many topological properties of the space $\mathcal{G}_m$ are discussed in \cite{Champ}. In this article, we will need some basic results from \cite{Champ}. The first result, describes the limits of convergent marked quotient groups.

\begin{theorem}
Suppose $\lim (G_i, S_i)=(G, S)$ and for any $i$, assume that $K_i$ is a normal subgroup of $G_i$. Assume that $\overline{S}_i$ is the canonical image of $S_i$ in $G_i/K_i$. If we have $\lim (G_i/K_i, \overline{S}_i)=(H, T)$, then $H=G/K$ for some normal subgroup $K$ and $T$ is the canonical image of $S$ in $G/K$.
\end{theorem}

In our main argument, we will give an explicit description of this normal subgroup $K$ in terms of the normal subgroups $K_i$. The second result, concerns the notion of fully residualness. Let $\mathfrak{X}$ be a class of groups. We say that a group $G$ is fully residually $\mathfrak{X}$, if for any finite subset $E\subseteq G$, there exists a group $H\in \mathfrak{X}$ and a homomorphism $\alpha:G\to H$ such that the restriction of $\alpha$ to $E$ is injective.

\begin{theorem}
Any finitely generated residually $\mathfrak{X}$-group is a limit of a sequence of marked groups from $\mathfrak{X}$. Conversely, any finitely presented limit of such marked groups is fully residually $\mathfrak{X}$.
\end{theorem}

To explain the next result from \cite{Champ}, we need some logical concepts. Let $\mathcal{L}=(1, ^{-1}, \cdot)$ be the first order language of groups. For a group $G$, we denote by $\mathrm{Th}(G)$, the first order theory of $G$, i.e. the set of all first order sentences in the language $\mathcal{L}$ which are true in $G$. The universal theory of $G$ is denoted by $\mathrm{Th}_{\forall}(G)$ and consists of all elements of $\mathrm{Th}(G)$ which have just universal quantifiers in their normal form.

\begin{theorem}
Suppose a sequence $(G_i, S_i)$ of marked groups converges to $(G, S)$. Then we have $\limsup_i \mathrm{Th}_{\forall}(G_i)\subseteq \mathrm{Th}_{\forall}(G)$. Conversely, if $\bigcap_i \mathrm{Th}_{\forall}(G_i)\subseteq \mathrm{Th}_{\forall}(G)$, then for any marking $(G, S)$, there exits a sequence of integers $(n_i)$ and subgroups $H_i\leq G_{n_i}$ such that a sequence of suitable markings of $H_i$s converges to $(G, S)$.
\end{theorem}

There is also a logical connection between convergence in the space of marked groups and ultra-products. Let $(G_i)$ be a sequence of groups and $\mathcal{U}$ be an ultra-filter on $\mathbb{N}$. Define a congruence over $\prod_iG_i$ by
$$
(x_i)\sim (y_i) \Leftrightarrow \{i:\ x_i=y_i\}\in \mathcal{U}.
$$
The quotient group $\prod_iG_i/\sim$ is called the ultra-product of the groups $G_i$ with respect to $\mathcal{U}$. We denote this new group by $\prod_iG_i/\mathcal{U}$.  A special case of the well-known theorem of L\v{o}s says that
$$
\mathrm{Th}_{\forall}(\prod_iG_i/\mathcal{U})=\lim_{\mathcal{U}}\mathrm{Th}_{\forall}(G_i),
$$
where, $\lim_{\mathcal{U}}$ of any sequence of sets $(A_i)$ is the set of all elements which belong to $\mathcal{U}$-almost all number of $A_i$s.

\begin{theorem}
Let $\lim (G_i, S_i)=(G, S)$. Then $G$ can be embedded in an ultra-product $\prod_iG_i/\mathcal{U}$, for some ultra-filter $\mathcal{U}$. Conversely, let $G$ be any finitely generated subgroup of some ultra-product $\prod_iG_i/\mathcal{U}$. Then for any marking $(G, S)$, there exits a sequence of integers $(n_i)$ and subgroups $H_i\leq G_{n_i}$ such that a sequence of suitable markings of $H_i$s converges to $(G, S)$.
\end{theorem}

\section{Main result}
We work within the space of marked groups $\mathcal{G}_m$. In \cite{Guyot}, Guyot determined the structure of limits of dihedral groups. The main result of \cite{Guyot} is the following.

\begin{theorem}
Let $G$ be a non-abelian finitely generated group. Then the following conditions are equivalent:\\

1- $G$ is a limit of dihedral groups.

2- $G$ is fully residually dihedral.

3- $G$ is isomorphic to a semidirect product $\mathbb{Z}_2\ltimes A$, where $A$ is a finitely generated abelian group with a cyclic torsion part, such that $\mathbb{Z}_2$ acts by multiplication by $-1$.

4- $\bigcap_{n\geq 3}\mathrm{Th}_{\forall}(D_{2n})\subseteq \mathrm{Th}_{\forall}(G)$.

5- $G$ can be embedded in some ultra-product of dihedral groups.
\end{theorem}

Our aim is to give the same characterization for the case of quaternion groups. Recall that by a quaternion group, we mean in fact a generalized quaternion group.

\begin{theorem}
Let $G$ be a non-abelain finitely generated group. Then the following conditions are equivalent:\\

1- $G$ is a limit of quaternion  groups.

2- $G$ is fully residually quaternion.

3- $G$ is isomorphic to a group of the form
$$
\frac{\mathbb{Z}_4\ltimes (\mathbb{Z}^l\oplus \mathbb{Z}_{2^k})}{\langle (2, 2^{k-1})\rangle},
$$
for some integers $l$ and $k$, such that the action of $\mathbb{Z}_4$ on $\mathbb{Z}^l\oplus \mathbb{Z}_{2^k}$ is given by $x\cdot a=(-1)^xa$.

4- $\bigcap_{n\geq 3}\mathrm{Th}_{\forall}(Q_{2^n})\subseteq \mathrm{Th}_{\forall}(G)$.

5- $G$ can be embedded in some ultra-product of quaternion groups.
\end{theorem}

\begin{proof}
Our pattern for the proof is the following:
$$
1\Rightarrow 5\Rightarrow 4\Rightarrow 1, \ \ 1 \Leftrightarrow 3, \ \ 1 \Leftrightarrow 2
$$

$(1 \Rightarrow 5)$. Let $G$ be a limit of quaternion groups. Then by Theorem 4, there exists an ultra-filter $\mathcal{U}$ such that $G$ embeds in $\prod_{n\geq 3}Q_{2^n}/\mathcal{U}$.\\

$(5 \Rightarrow 4)$. Suppose for some ultra-filter $\mathcal{U}$, we have $G\leq \prod_{n\geq 3}Q_{2^n}/\mathcal{U}$. Then,
$$
\mathrm{Th}_{\forall}(\prod_{n\geq 3}Q_{2^n}/\mathcal{U})\subseteq \mathrm{Th}_{\forall}(G).
$$
By the theorem of L\v{o}s, we have
\begin{eqnarray*}
\mathrm{Th}_{\forall}(\prod_{n\geq 3}Q_{2^n}/\mathcal{U})&=& \lim_{\mathcal{U}}(\mathrm{Th}_{\forall}(Q_{2^n}))\\
                                                         &\supseteq&\bigcap_{n\geq 3}\mathrm{Th}_{\forall}(Q_{2^n}),
\end{eqnarray*}
and so $4$ follows.

$(4 \Rightarrow 1)$. By Theorem 3, there exists a sequence $(n_i)$ of integers and subgroups $H_i\leq Q_{2^{n_i}}$ such that for suitable markings, we have $\lim (H_i, T_i)=(G, S)$. Every $H_i$ is cyclic or quaternion. If almost all $H_i$ are cyclic then $G$ is abelian,  which is not the case. Because of convergence, almost all $H_i$ are quaternion and $1$ follows.

$(1 \Rightarrow 3)$. Suppose that $G$ is a limit of quaternion groups. Then for suitable markings, we have
$$
\lim (Q_{2^i}, T_i)=(G, T).
$$
Recall that $Q_{2^i}=(\mathbb{Z}_4\ltimes \mathbb{Z}_{2^{i-1}})/K_i$, where $K_i=\langle(2, 2^{i-2})\rangle$. Let
$$
T_i=(a_{i1}K, \ldots, a_{im}K),\ \ t_i=(2, 2^{i-2}).
$$
Then $S_i=(a_{i1}, \ldots, a_{im}, t_i)$ is a generating set for $\mathbb{Z}_4\ltimes \mathbb{Z}_{2^{i-1}}$. We have $(\mathbb{Z}_4\ltimes \mathbb{Z}_{2^{i-1}}, S_i)\in \mathcal{G}_{m+1}$ and since $\mathcal{G}_{m+1}$ is compact, so a subsequence of this later sequence is convergent, i.e. there exists a sequence $(n_i)$ and a marked group $(H, S)\in \mathcal{G}_{m+1}$, such that
$$
\lim (\mathbb{Z}_4\ltimes \mathbb{Z}_{2^{n_i-1}}, S_{n_i})=(H, S).
$$
On the other hand, in $\mathcal{G}_{m+1}$, we have $\lim (Q_{2^{n_i}}, T_{n_i}+1)=(G, T+1)$, where $T+1$ denotes $T$ extended by one extra $1$ from right (and similarly $T_{n_i}+1$). By Theorem 1, we see that $G=H/K$,  for some normal subgroup $K$. Before computing $H$, we show that $K\subseteq Z(H)$. For simplicity, we put $H_i=\mathbb{Z}_4\ltimes \mathbb{Z}_{2^{n_i-1}}$ and $S_i=S_{n_i}$. Suppose that
$$
N_i=\{ w\in \mathbb{F}_{m+1}: w(S_i)=1\},\ \ N=\{ w\in \mathbb{F}_{m+1}: w(S)=1\}.
$$
Then we have\\

1- $H_i\cong \mathbb{F}_{m+1}/N_i$ and $H\cong \mathbb{F}_{m+1}/N$.

2- $N=\liminf_i N_i$.\\

Similarly, we know that the marked group $(H_i/K_i, T_i+1)$ is corresponding to a normal subgroup $M_i$ in $\mathbb{F}_{m+1}$. We have
\begin{eqnarray*}
M_i&=&\{ w\in \mathbb{F}_{m+1}: w(T_i+1)=1\}\\
   &=&\{ w\in \mathbb{F}_{m+1}: w(S_i)\in K_i\}.
\end{eqnarray*}
By a similar argument, $(H/K, T+1)$ corresponds to
$$
M=\{ w\in \mathbb{F}_{m+1}: w(S)\in K\}.
$$
Therefore, we have
$$
\{ w\in \mathbb{F}_{m+1}: w(S)\in K\}=\liminf_i \{ w\in \mathbb{F}_{m+1}: w(S_i)\in K_i\}.
$$
Note that this description of $K$ is generally true for all cases of Theorem 1. Recall that $K_i\subseteq Z(H_i)$. We now can show that $K\subseteq Z(H)$. Let $a\in K$ and $b\in H$. There are words $w$ and $v$ such that $a=w(S)$ and $b=v(S)$. Moreover $w\in M$. So, there is a $j_0$ such that for all $i\geq j_0$, $w(S_i)\in K_i$. So, for $i\geq j_0$, we have $[w(S_i), H_i]=1$. As a special case $[w(S_i), v(S_i)]=1$. Let $R$ be the length of the commutator word $[w, v]$. Since $(H_i, S_i)\to (H, S)$, so there is $j_1$, such that for all $i\geq j_1$, two closed balls $B_R(H_i, S_i)$ and $B_R(H, S)$ are marked isomorphic. Let $j=\max\{ j_0, j_1\}$. Then for $i\geq j$,
$$
B_R(H_i, S_i)\cong B_R(H, S),\ \ [w(S_i), v(S_i)]=1.
$$
Hence, we have also $[w(S), v(S)]=1$ and this shows that $a\in Z(H)$.

It remains to determine the structure of $H$. But this is completely similar to the process of finding limits of dihedral groups in \cite{Guyot}. Hence, we know that  $H=\mathbb{Z}_4\ltimes A$, where $A$ is a finitely generated abelian group with a cyclic torsion part. As
$$
\bigcap_{n\geq 2}\mathrm{Th}_{\forall}(\mathbb{Z}_4\ltimes \mathbb{Z}_{2^n})\subseteq \mathrm{Th}_{\forall}(\mathbb{Z}_4\ltimes A),
$$
we see that for all odd prime $p$, the universal sentence
$$
\forall x (x^p=1\to x=1)
$$
which is true in all groups $\mathbb{Z}_4\ltimes \mathbb{Z}_{2^n}$, is already true in $\mathbb{Z}_4\ltimes A$. This shows that in the later group, the torsion part  is a cyclic $2$-group. Hence for some integers $l$ and $k$, we have $A=\mathbb{Z}^l\oplus \mathbb{Z}_{2^k}$. By a simple computation in semidirect product, we see that
$$
Z(H)=\{ (0,0), (2,0), (0, 2^{k-1}), (2, 2^{k-1})\}.
$$
Therefore we have five alternatives for $K$:
\begin{eqnarray*}
K&=&\{ (0,0)\},\\
K&=&\{ (0,0), (2,0)\},\\
K&=&\{ (0,0), (0, 2^{k-1})\},\\
K&=&\{ (0,0), (2, 2^{k-1})\},\\
K&=&Z(H).
\end{eqnarray*}
We will prove that the only acceptable case is $K=\{ (0,0), (2, 2^{k-1})\}$. Recall that all quaternion groups have a unique involution. This fact can be translated into a universal sentence as
$$
\forall x, y(x^2=y^2=1\to (x=1\vee y=1\vee x=y)).
$$
Since $G$ is a limit of quaternion groups, so the above sentence is also true in $G$, i.e. $G$ has a unique involution. Computation in semidirect product, reveals the following facts:\\

i- in the case one, there are at least three involutions
$$
(2,0)K, (0, 2^{k-1})K, (2, 2^{k-1})K.
$$

ii- in the second case, there are at least two involutions
$$
(1, 0)K, (0, 2^{k-1})K.
$$

iii- in the third case, there is also at least two involutions
$$
(2,0)K, (0, 2^{k-2})K.
$$

iv- in the case five, there are at least two involutions
$$
(1,0)K, (0, 2^{k-2})K.
$$

It remains only the case four where actually the resulting group has a unique involution. Summarizing, we conclude that $G$ is isomorphic to a group of the form
$$
\frac{\mathbb{Z}_4\ltimes (\mathbb{Z}^l\oplus \mathbb{Z}_{2^k})}{\langle (2, 2^{k-1})\rangle},
$$
for some integers $l$ and $k$, such that the action of $\mathbb{Z}_4$ on $\mathbb{Z}^l\oplus \mathbb{Z}_{2^k}$ is given by $x\cdot a=(-1)^xa$.

$(3 \Rightarrow 1)$. We know that the abelian group $\mathbb{Z}^l\oplus \mathbb{Z}_{2^k}$ is a limit of cyclic $2$-groups, and hence $\mathbb{Z}_4\ltimes (\mathbb{Z}^l\oplus \mathbb{Z}_{2^k})$ is a limit of groups of the form $\mathbb{Z}_4\ltimes \mathbb{Z}_{2^{n-1}}$. By Theorem 1, we have
$$
\frac{\mathbb{Z}_4\ltimes \mathbb{Z}_{2^{n-1}}}{K_n}\to \frac{\mathbb{Z}_4\ltimes (\mathbb{Z}^l\oplus \mathbb{Z}_{2^k})}{L},
$$
for some normal subgroup $L$. Again checking the number of involutions, we conclude that $L=K$.

$(2 \Rightarrow 1)$. Let $G$ be fully residually quaternion. Suppose $S$ is an arbitrary generating set for $G$. For any $R>0$, the closed ball $B_R(G,S)$ is  finite. Hence, there is a $n\geq 3$ and a homomorphism $\alpha:G\to Q_{2^n}$, such that its restriction to $B_R(G, S)$ is injective. Let $T=\alpha(S)$. Then clearly we have
$$
d((G,S), (Q_{2^n}, T))\leq e^{-R}.
$$
This shows that $G$ is a limit of quaternion groups.

$(1 \Rightarrow 2)$. By 3, the group $G$ is finitely presented and hence by Theorem 2, it is fully residually quaternion.

\end{proof}

\end{document}